\documentclass[oneside,11pt]{article}
\usepackage{epsfig, amsfonts, amsmath, amsthm,amsbsy}
\usepackage{color}
\usepackage{subfigure}

\newtheorem{theorem}{Theorem}
\newtheorem{lemma}{Lemma}

\newtheorem{remark}{Remark}

\newtheorem*{example-non}{Example}

\def\ve{\varepsilon}
\def\vr{\varepsilon}
\def\ds{\displaystyle}
  
  \def\cbr{\color{red}}

\begin{document}

 \title{Parameter-uniform numerical methods for singularly perturbed linear transport problems\thanks{This research was partially supported  by the Institute of Mathematics and Applications (IUMA), the projects PID2019-105979GB-I00 and PGC2018-094341-B-I00 and the Diputaci\'on General de Arag\'on (E24-17R).}}
\author{J.L Gracia    \thanks {IUMA and Department of Applied Mathematics, University of Zaragoza, Spain;   email: {jlgracia@unizar.es} }    \and A. Navas-Montilla   \thanks {Centro Universitario de la Defensa, Zaragoza, Spain; email: anavas@unizar.es }  \and
        E. O'Riordan \thanks { School of Mathematical Sciences, Dublin City
University, Dublin 9, Ireland;
 email: {eugene.oriordan@dcu.ie}
}
}

\maketitle

\begin{abstract}

Pointwise accurate numerical methods are constructed and analysed for three classes of singularly perturbed first order transport problems. The methods involve piecewise-uniform Shishkin meshes and the numerical approximations are shown to be parameter-uniformly  convergent in the maximum norm. A transport problem from the modelling of fluid-particle interaction is formulated and used as a test problem for these numerical  methods. Numerical results are presented to illustrate the performance of the numerical methods and to confirm the theoretical error bounds established in the paper.
\
\end{abstract}


\section{Introduction}

Mathematical models of flow related physical problems (e.g. semiconductor devices \cite[Chapter 3]{markowich}, plasma sheath formation
\cite{slemrod}  or fluid-particle interaction \cite{baranger,Lagoutiere,astro}) typically involve coupled systems of three of more partial differential equations. Transport equations of the form $u_t + \nabla f(u) = g(x,t)$ often appear in one or more of the equations within these systems. In many cases, the flow is convection-dominated and  steep  gradients appear in certain subregions of the domain. These problems are singularly perturbed problems. Our interest lies in designing parameter-uniform numerical methods \cite{fhmos} for a wide class of singularly perturbed problems.
These numerical methods are designed to generate globally pointwise accurate approximations to the solutions and the order of convergence is retained irrespective of the value of the singular perturbation parameter. In this article, we study singularly perturbed problems arising in linear transport problems of the form $u_t+au_x+bu=f$.

In \cite{shish}, Shishkin examined singularly perturbed transport  problems of the form: Find $u(x,t)\in C^2(\bar \Omega), \Omega :=(0,L] \times (0,T]$ such that
\begin{align*}
&u_t + \ve a u_x +bu = f(x,t), (x,t) \in \Omega;  \\
& u(x,t) \text{ given for }  (x,t) \in \partial \Omega := \bar \Omega \setminus \Omega,\quad
a >0,  \, b\ge 0, \,   (x,t) \in \bar \Omega.
\end{align*}
The parameter $0 <\ve \leq 1$ can be arbitrarily small in size.
As $\ve \rightarrow 0$, the characteristic curves (associated with this first order problem) tend towards vertical lines of the form $x=x_0$.  A boundary layer of order $O(\ve)$ will appear near the edge $x=0$. Using an appropriate piecewise uniform mesh and a standard finite difference operator, Shishkin
\cite{shish}  establishes a first order parameter-uniform \cite{fhmos} numerical method for this transport problem.

In the context of  two-parameter singularly perturbed elliptic problems \cite[\S 3 and Remark 8.3]{maria}, first order problems of the form
\[
\mu \vec a \cdot \nabla u -bu = f(x,y), \quad (x,y) \in [0,1)^2, \ u(x,0), u(0,y) \quad \hbox{given};
\]
were examined. Boundary layers of order $O(\mu)$ appear near the edges $x=0, y=0$. (If $u(x,1), u(1,y)$  are given instead of $u(x,0), u(0,y)$, the layers will form along $x=1$ and $y=1$.)
Using an appropriate Shishkin mesh \cite{maria}, a parameter-uniform error bound of the form $(N^{-1} \ln N)$ can be established.

In this paper, we examine linear singularly perturbed transport problems of the form
\begin{align}\label{transport}
 u_t + a u_x +bu = f(x,t;\ve) ;\quad
u(0,t) = \psi (t; \ve) ; \, u(x,0) = \phi (x;\ve),
\end{align}
where layers (boundary and/or interior) are generated by the fact that layers are present in one or more of the functions $\psi , \phi , f$. In contrast to the problems examined in  \cite{maria} and \cite{shish}, the characteristic curves associated with
(\ref{transport}) do not depend on the singular perturbation parameter. 
Moreover, under the assumption $a(x,t) \geq \alpha >0, \forall (x,t) $ the tangents to the characteristic curves always have a finite slope.

 The paper is structured as follows. In \S \ref{sec:non-singularly}  a comparison principle is given and the regularity of the solution is discussed for non-singularly perturbed transport problems. A result of convergence for a classical scheme is established when the solution  $u$ of problem~\eqref{transport} satisfies $u\in C^1(\bar \Omega)$.
 In \S \ref{sec:FlowTowards}, \S \ref{sec:pulse} and \S \ref{sec:FlowAway} three problem classes of singularly perturbed problems are considered. The boundary or interior layers are generated due to the forcing term, the initial or boundary conditions depending on the singular perturbation parameter. The interior layers are located along the characteristic curves associated with problem \eqref{transport} and the numerical methods will need to align the mesh to this characteristic curve in order to accurately track the interior layer~\cite{shish04}, which are described in  \S \ref{sec:pulse} and  the Appendix. We construct and analyse  parameter-uniform numerical methods for all the singularly perturbed transport  problems considered of the form given in (\ref{transport}). In \S \ref{sec:numerical}, we examine a problem motivated by a mathematical model of fluid-particle interaction in particle-laden flows \cite{astro}; the solution is decomposed into several components which are approximated using the algorithms described in the three previous sections. The numerical results indicate that our method generates almost first-order global approximations to the solution of this model. Finally, we draw some conclusions.

{\bf Notation:} Throughout the paper, $C$ denotes a generic constant that is independent of the singular perturbation parameter $\ve$ and all the discretization parameters.
The $L_\infty$ norm on a domain $D$ will be denoted by $\Vert \cdot \Vert _D$. If the norm is not subscripted, then $\Vert \cdot \Vert  = \Vert \cdot \Vert _{\bar \Omega}$.

\section{Non-singularly perturbed problem} \label{sec:non-singularly}

Consider first order differential  operators of  the form
\[
L \omega:=\omega_t+a(x,t) \omega_x  +b(x,t)\omega,
\]
and  the domain $\Omega :=(0,L] \times (0,T]$.
\begin{lemma}\label{minprin}
 Assume $\omega(x,0) \geq 0, \ 0 \leq x \leq L,  a(x,t)\geq 0, (x,t) \in \bar \Omega, a,b \in C^0(\bar \Omega)$ and $\omega \in C^1(\Omega ) \cap C^0(\bar \Omega)$.
 If $\omega(0,t) \geq 0$ or $\omega_x(0^+,t) <0$  for $ t\in (0 ,T]$ and $L\omega \geq 0$ in $\Omega $, then $\omega(x,t)  \geq 0, \forall (x,t) \in \bar \Omega$.

\end{lemma}
\begin{proof} The proof is by contradiction. Assume $\omega < 0$ within $\bar \Omega$.  Let $\nu(x,t) = \omega(x,t) e^{-\beta t}, \beta >  \Vert b \Vert $ and
$
\nu(p,q) := \min _{\bar \Omega} \nu(x,t) < 0.
$
By our assumptions, $q \neq 0,  p \neq 0$. Hence $(p,q) \in \Omega $ where $\nu_x(p^-,q) \leq 0, \nu_t(p,q^-) \leq 0$. However, then
\[
L\omega (p,q) = e^{\beta  q}(\nu_t+a\nu_x+(b+\beta )  \nu)(p,q) <0,
\]
which is a contradiction.
\end{proof}

Consider the  initial-boundary value problem: Find $u$ such that
\begin{subequations}\label{main-problem}
\begin{align}
&Lu= f(x,t), \ (x,t) \in \Omega, \quad u(0,t) = \psi (t),  \  u(x,0)= \phi (x);  \\
& a(x,t) \geq \alpha > 0, \ b(x,t) \geq 0, \ (x,t) \in \bar \Omega ;\\
& \psi  \in C^2([0,T]), \phi  \in C^2([0,L]), \quad a,b,f\in C^2(\bar \Omega).
\end{align}
\end{subequations}
There is no loss in generality in assuming that $b \geq 0$, as we simply use the transformation $u^T = u e^{-\beta t}, \beta >  \Vert b \Vert $
and solve $Lu^T + \beta u^T = e^{\beta t}f$ for $u^T$.
From Bobisud \cite[Appendix]{bobisud},  in order for $u \in C^2(\bar \Omega)$, we require the following compatibility conditions
on the data:
\begin{subequations}\label{compat}
 \begin{align}
&\phi (0) = \psi (0); \label{comp0}\\
&\psi '(0) +a(0,0) \phi ' (0) +b(0,0)  \psi (0) = f(0,0); \label{comp1}\\
&\bigl[ f_t(0,0) - b_t(0,0) \psi (0) -b(0,0) \psi '(0) - \psi ''(0) \bigr]  \nonumber  \\
&\hspace{1cm}+\phi '(0)\bigl[ a(0,0)(b(0,0)  +a_x(0,0)) -a_t(0,0) \bigr] +a^2(0,0) \phi ''(0) \nonumber \\
& \hspace{1cm}= a(0,0)\bigl[ f_x(0,0) - b_x(0,0) \psi (0) \bigr]. \label{comp2}
\end{align}
\end{subequations}
As in \cite{bobisud}, define the characteristic curve passing through the point $(s, \tau)$ by $x=g(t;s,\tau )$, where  $x=g(t)$ is the solution of the initial value problem
\[
\frac{dx}{dt} =a(x,t), \qquad x(\tau)= s.
\]
Define  $\Gamma (s,\tau ) $ such that
$
g(\Gamma (s,\tau);s,\tau )=0.$
The point $(x,t)=(0, \Gamma(s,\tau))$ is where the characteristic curve through the point $(s,\tau)$ intercepts the axis $x=0$.
Define the two subregions $D_1:= \{ (x,t) \, \vert  \, x >  g(t;0,0) \} , D_2:= \{ (x,t) \, \vert \, x< g(t;0,0) \}$.
The exact solution of problem (\ref{main-problem}) is then given by  \cite{bobisud}
\begin{subequations}\label{exact}
\begin{align}
 u(s,\tau ) &=
\phi (g(0;s,\tau)) B(0,\tau ; s, \tau) + \nonumber \\
& + \int _{z=0}^{\tau } f(g(z; s,\tau),z)B(\tau,z ; s, \tau)  \ dz, \qquad (s,\tau ) \in D_1;  \\
  u(s,\tau ) &=
\psi (\Gamma (s,\tau ))  B(\Gamma (s,\tau ) , \tau ; s, \tau) +\nonumber \\
&+ \int _{z=\Gamma (s,\tau ) }^{\tau } f(g(z; s,\tau),z) B(\tau,z ; s, \tau)   \ dz, \  (s,\tau ) \in D_2;
\end{align}
\end{subequations}
where
\[
B(\eta ,\iota;s, \tau) := e^{- \int _{t=\eta}^ \iota b(g(t;s,\tau),t)  \ dt }.
\]
If (\ref{comp0}) is not satisfied, then the solution  of problem (\ref{main-problem}) is discontinuous across the  characteristic curve $ x= g(t;0,0)$. If (\ref{comp1})  is not satisfied, then the solution $u$  of problem (\ref{main-problem}), (\ref{comp0}) is such that $u \in  C^0(\bar \Omega) \setminus C^1(\bar \Omega)$. Finally, if (\ref{comp2})  is not satisfied, then the solution  $u$ of problem (\ref{main-problem}),  (\ref{comp0}), (\ref{comp1}) is such that $u \in  C^1(\bar \Omega) \setminus C^2(\bar \Omega)$. We associate the following four functions $\{ \chi _i \} _{i=0}^3$, with the compatibility conditions (\ref{compat}):
\begin{subequations}\label{singular}
\begin{align}
L\chi _0(x,t) & =0, \qquad  L\chi _1(x,t) =1, \quad (x,t) \in \Omega, \\
L\chi _2(x,t) &=x,\qquad L\chi _3(x,t) =t,\quad (x,t) \in \Omega, \\
 \chi _0 (0,t) &=0, \ t >0,\quad  \ \chi _0(x,0) =1, \  x\geq 0,  \\
 \chi _i (0,t) & =0, \ t >0, \quad \chi _i(x,0) =0, \ x\geq 0, \ i=1, 2,3;
\end{align}
\end{subequations}
where $\chi _0 \not \in C^0(\bar \Omega), \chi _1 \in  C^0(\bar \Omega) \setminus C^1(\bar \Omega)$ and $\chi _2, \chi _3 \in  C^1(\bar \Omega) \setminus C^2(\bar \Omega)$.
Explicit representations for $\chi _0$ and $\chi _1$ can be  determined using (\ref{exact}).  In general, for the solution $u$ of problem (\ref{main-problem}) we note that
\begin{align*}
&u(x,t) - \bigl (\psi (0) + (\phi (0) -\psi (0) ) \chi _0(x,t) \bigr) \in C^0(\bar \Omega) \setminus C^1(\bar \Omega); \\
&u(x,t) - \bigl (\psi (0) + (\phi (0) -\psi (0) ) \chi _0(x,t) \bigr)  - A\chi _1(x,t) \in C^1(\bar \Omega) \setminus C^2(\bar \Omega); \\
&\hbox{where} \qquad A:= f(0,0) - \bigl( \psi '(0) +a(0,0) \phi ' (0) +b(0,0)  \psi (0) \bigr).
\end{align*}

After  separating off the singular terms $\chi _0$ and  $\chi _1$ we can use a simple numerical method to generate an approximation to the solution of a problem of the form (\ref{main-problem}),  (\ref{comp0}), (\ref{comp1}). Select a set of mesh points \[
\bar \Omega ^{N,M} := \left \{ (x_i,t_j) \, \vert \, x_0=0, \, t_0=0, \, x _N=L, \, t_M=T \right \}_{i,j=0}^{N,M},
\]
where   the mesh steps are denoted by $h_i:=x_i-x_{i-1}, 1\leq i \leq N; \, k_j:= t_j-t_{j-1}, 1 \leq j \leq M$ and $\max _i h_i \leq CN^{-1}, \max _j k_j\leq CM^{-1}$. Define the  set of interior mesh points as $\Omega ^{N,M} := \{ (x_i,t_j)\}_{i,j=1}^{N,M}$. We can discretize problem  (\ref{main-problem}) using an upwinded finite difference operator\footnote{We use the following notation for the finite difference operators:
\[
D^-_t Y (x_i,t_j) := \ds\frac{Y(x_i,t_j)-Y(x_i,t_{j-1})}{k_j}, \quad
D^-_x Y(x_i,t_j) :=\ds\frac{Y(x_i,t_j)-Y (x_{i-1},t_j)}{h_i}. \]
} of the form:
\begin{subequations}\label{discrete-main-problem}
\begin{align}
L^{N,M} U(x_i,t_j) &= f(x_i,t_j), \quad (x_i,t_j) \in \Omega ^{N,M},\\
U(0,t_j) & = u(0,t_j), \, t _j >0, \quad U(x_i,0) = u(x_i,0), \, x_i \geq 0;
\\
\hbox{where }  L^{N,M} U(x_i,t_j) & := (D^-_t+a(x_i,t_j) D^-_x + b(x_i,t_j)I)U(x_i,t_j) .
\end{align}
\end{subequations}
 Similarly to the continuous problem, the discrete operator $L^{N,M}$ satisfies a discrete comparison principle.

We form a global approximation $\bar U$ to the solution of (\ref{main-problem}),  (\ref{comp0}), using
\[
\bar U(x,t) : = \sum _{i=0,j=1}^{N,M} U(x_i,t_j) \varphi _i(x)  \eta_j(t),
\]
where $\varphi _i(x)$ is the standard hat function centered at $x=x_i$ and $ \eta _j(t) :=(t-t_{j-1})/k, \, t \in [t_{j-1},t_j), \,  \eta _j(t) :=0, \, t \not \in [t_{j-1},t_j)$.
\begin{theorem}\label{Result0}
If $U$ is the solution of (\ref{discrete-main-problem}) and $u \in C^1(\bar \Omega)$ is the solution of (\ref{main-problem}) then
\[
\Vert \bar U -u \Vert  \leq  C (N^{-1}+ M^{-1}).
\]
\end{theorem}
\begin{proof}
 As $\frac{d x}{dt} >0$, then at each time level $t=t_*$, there is only one point $(x_*,t_*)$
where the characteristic curve $x= g(t;0,0)$ cuts the line $t=t_*$. From the explicit expression in (\ref{exact}), we see that the second  order derivatives of $\chi _2$
and $\chi _3$ are continuous at all points except for the points
on the curve $x=g(t_*;0,0)$. Moreover, the second derivatives are bounded along this characteristic curve  $\Upsilon=\{(x_*,t_*)=(g(t_*;0,0),t_*), \ 0 \le t_* \le T\}$. That is,
\[
\vert  u_{xx}  (x^\pm _*,t_*)  \vert  \leq C, \quad  \Vert u_{xx}  \Vert _{\Omega \setminus { \Upsilon}} \leq C,  \quad
\vert u_{tt}    (x_*,t^\pm _*)  \vert  \leq C, \quad \Vert u_{tt}   \Vert _{\Omega \setminus { \Upsilon}} \leq C.
\]
For each $t=t_j$, there exists a unique subinterval $[x_{I-1},x_I]$ such that $ x_*^I\in (x_{I-1},x_I]$ and $ x_*^I=g(t_j)$.
For this particular interval,
\begin{align*}
\left (\frac{\partial} {\partial x} - D_x^- \right) u(x_i,t_j)&= \frac{1}{h_i} \int _{s=x_{i-1}}^{x_i} u_x(x_i,t_j) - u_x(s,t_j) \ ds \\
 &= \frac{1}{h_i} \int _{s=x_{i-1}}^{x_i} \left(  \int _{r=s}^{ x_*^I}  u_{xx}(r,t_j) \ dr  +  \int _{r= x_*^I}^{x_i}u_{xx}(r,t_j) \ dr\right) ds.
\end{align*}
In an analogous fashion, for each $x=x_i$, there exists a unique subinterval $[t_{J-1},t_J]$ such that $ t_*^J \in(t_{J-1},t_J]$ and $x_i=g( t_*^J )$.
For this particular interval,
\begin{align*}
\left (\frac{\partial} {\partial t} - D_t^-\right) u(x_i,t_j)
 &=& \frac{1}{k_j} \int _{s=t_{j-1}}^{t_j} \left(  \int _{r=s}^{{ t_*^J}}  u_{tt}(x_i,r) \ dr  +  \int _{r={ t_*^J}}^{t_j}u_{tt}(x_i,r) \ dr\right) ds.
\end{align*}
 Hence, inside of these intervals we have
\[
\left \vert \left (\frac{\partial} {\partial x} - D_x^- \right) u(x_i,t_j) \right \vert \le C N^{-1},
\quad
\left \vert \left (\frac{\partial} {\partial t} - D_t^-\right) u(x_i,t_j) \right \vert \le C M^{-1}.
\]
On the other hand, truncation error estimates outside of these particular subintervals can be easily obtained using that $u_{xx}$ and $u_{tt}$ are continuous functions. Therefore, we have
\[
\bigl \vert  L^{N,M}(u-U)(x_i,t_j) \bigr \vert \leq CN^{-1} +CM^{-1}, \qquad (x_i,t_j) \in \Omega ^{N,M}.
\]
As the upwind operator $L^{N,M}$ satisfies a discrete  comparison principle, it follows that
\[
\bigl \vert  (u-U)(x_i,t_j) \bigr \vert \leq CN^{-1} +CM^{-1}.
\]
This nodal error bound extends to a global error bound using standard interpolation error estimates.
\end{proof}
\begin{remark} If instead of separating away the term involving $\chi _1 (x,t)$,  one could  use the numerical method to generate an approximation, when $u \in C^0(\bar \Omega ) \setminus C^1(\bar \Omega ) $.  In this case, unlike Theorem \ref{Result1}, the computed order of convergence in our numerical experiments reduce to $0.5$.
\end{remark}

In the remaining sections, we construct numerical methods in the case where the problem (\ref{main-problem}) is singularly perturbed.

\section{\bf Flow towards an attractive  force} \label{sec:FlowTowards}

In the first problem class to be examined in this paper, the forcing term has a layer at the outflow ($x=L$)  of the domain.
Consider problem (\ref{main-problem}), (\ref{compat})  with the additional conditions

\begin{subequations}\label{flow-towards}
\begin{align}
 \left \vert \frac{\partial ^{i+j}}{\partial x ^i  \partial t^j } f(x,t;\ve) \right \vert & \leq  \frac{C_1}{\ve ^{1+i}} e^{\frac{x-L}{\ve}}; \quad 0 \leq i +j \leq 2, \\
 \frac{\partial ^{i+j}f}{\partial x ^i  \partial t^j } (0,0) & =0,  \quad   0 \leq i +j \leq 1 .
\end{align}
\end{subequations}
Then, the solution $u$ can be decomposed into the sum of a regular component $v \in C^2(\bar \Omega)$ and a layer component $w\in C^2(\bar\Omega)$ defined by
\begin{subequations}\label{decomp}
\begin{align}
Lv = 0, \quad (x,t)  \in \Omega; \quad v=u,\ (x,t) \in \partial \Omega := \bar \Omega \setminus \Omega ;\label{towards-reg}\\
Lw = f, \quad (x,t)  \in\Omega;\quad w=0,\ (x,t) \in \partial \Omega.\label{towards-layer}
\end{align}
\end{subequations}
In this section the  problem data $a,b, \psi , \phi $,  are smooth functions that do not depend on the singular perturbation parameter $\ve$.
By Lemma \ref{minprin},
\[
\vert u(x,t) \vert \leq \Vert \phi \Vert _{[0,T]} + \Vert \psi \Vert _{[0,L]}+ \frac{C_1}{\alpha} e^{\frac{ x-L}{\ve}}.
\]
\begin{lemma}\label{Lemma2} For all $(x,t) \in \bar \Omega$, the components in (\ref{decomp}) satisfy
\begin{subequations}\label{decomp-bounds}\begin{align}
\left\Vert \frac{\partial ^{i+j} v }{\partial x^i \partial t^j} \right \Vert & \leq C, \quad 0 \leq i+j \leq 2; \\
\left\vert w (x,t) \right \vert & \leq C e^{\frac{ x-L}{\ve}} \left(1-e^{\frac{ -\Vert a \Vert t}{\ve}} \right); \\
\left\vert \frac{\partial ^{i+j} w (x,t)}{\partial x^i \partial t^j } \right \vert & \leq C \ve ^{-(i+j)}e^{\frac{ x-L}{\ve}},\quad  1 \leq i+j \leq 2.
\end{align}\end{subequations}
\end{lemma}
\begin{proof} As the data for the problem (\ref{towards-reg}) are independent of the parameter $\ve$ and since $v \in C^2(\bar\Omega)$ the bounds on the derivatives of $v$ follow immediately.

Using Lemma \ref{minprin} with the obvious barrier function yields the bound on $w(x,t)$.
Since $w_x \in C^1(\bar \Omega)$, it satisfies the first order problem
\begin{equation}\label{w_x}
Lw_x +a_xw_x = f_x -b_x w, \quad w_x(x,0)=0, \ w_x(0,t) = \frac{f(0,t)}{a(0,t)}.
\end{equation}
Using $w_1:= e^{-\kappa t} w_x, \ \kappa := \min \{b+a_x,0 \}$ and Lemma \ref{minprin} again, we establish the bound
\[
 \vert w_x (x,t)  \vert \leq C \ve ^{-1}e^{\frac{ x-L}{\ve}} \quad \hbox{and then} \quad \vert w_t (x,t)  \vert \leq C \ve ^{-1}e^{\frac{ x-L}{\ve}}.
\]

Using (\ref{exact}) a closed form  representation of
the solution of (\ref{w_x}) exists.
Differentiating this  expression (w.r.t. $ x$) and using
$g(z;x,t) \leq \Vert a \Vert  (z-t)+ x$ we have that
\begin{eqnarray*}
\left \vert \frac{\partial ^2 w}{\partial x^2}(x,t) \right \vert \leq C
\begin{cases}
\int _{s=0}^{t}  \ve ^{-3}e^{\frac{\Vert a \Vert (s-t) +x-L}{\ve}} \ ds, & (x,t) \in D_1,  \\
(\int _{s=0}^{t} +  \int _{s=0}^{\Gamma (x,t)}) \left \vert \frac{\partial (f_x-b_xw) (g(s; x,t),t)}{\partial x} \right \vert \ ds, & (x,t) \in D_2.
\end{cases}
\end{eqnarray*}
Hence
\begin{eqnarray*}
\left \vert \frac{\partial ^2w}{\partial x^2}(x,t) \right \vert \leq C \ve ^{-2}e^{\frac{ x-L}{\ve }}.
\end{eqnarray*}
Use the differential equation $u_{tx} = (f -bu-au_x)_x $ to derive bounds on $u_{tx}$ and   $u_{tt} = (f -bu-au_x)_t $ to derive bounds on $u_{tt}$.
\end{proof}
The layer function can be further decomposed into
\begin{subequations}\label{extra-decomp}
\begin{align}
&w(x,t) = w_0(x,t)  + w_1(x,t), \qquad \hbox{where} \\
&L w_0=0,(x,t) \in \Omega, \quad w_0(0,t)=0,\  w_0(x,0) =\omega (x), \\
&a(x,0) \frac{d\omega}{dx} +b(x,0)\omega=f(x,0), \ 0 < x \leq L, \quad \omega (0)=0.
\end{align}
\end{subequations}
Observe that, for $0 <p<1$,
\begin{align*}
\left \vert \frac{d^{i} \omega (x)}{dx^i } \right \vert & \leq C \ve ^{-i}e^{\frac{ x-L}{\ve}}\ i =0,1,2; \
 \left \vert \frac{\partial ^{j}}{ \partial t^j }w_0 (x,t) \right \vert \leq C\ve ^{-j}e^{\frac{ x-L}{\ve}} e^{-\frac{p\alpha t}{\ve}},\  j =0,1,\\
L \left ( \frac{\partial w_1}{\partial t} \right) & =  f-a_t \frac{\partial w_1}{\partial x} -b_t w_1, \quad \frac{\partial w_1}{\partial t}(0,t) = \frac{\partial w_1}{\partial t}(x,0) = 0.
\end{align*}
Using the arguments in the previous  Lemma~\ref{Lemma2} we can establish that
\begin{equation}\label{sharper}
\left\vert \frac{\partial ^{j} w_1 (x,t)}{ \partial t^j } \right \vert\leq C \ve ^{1-j}e^{\frac{ x-L}{\ve}},\quad  j=1,2.
\end{equation}

To capture the layer near $x=L$, we define a piecewise-uniform Shishkin mesh \cite{fhmos}  in both space and time, which we denote by $\bar \Omega ^N_S$.
 The space domain
is split by $[0,L]= [0, L-\sigma ]\cup [L-\sigma ,L]$ and the time domain is subdivided into
$[0,T]= [0, \tau ]\cup [\tau ,T]$. In each coordinate direction, half of the mesh points are uniformly distributed in each subinterval. The transition points are
defined to be \[
\sigma := \min \left \{ \frac{L}{2}, \ve\ \ln {N} \right \} \quad \hbox{and} \quad
\tau := \min \left \{ \frac{T}{2},  C_2 \ve   \ln M \right \}, \ C_2 > \alpha ^{-1}.
\]
The two mesh steps  in space and time  are denoted by
\[
 H=2\frac{L-\sigma}{N}, \  h=2\frac{\sigma}{N}, \quad  K=2\frac{T-\tau}{M}, \  k=2\frac{\tau}{M}.
\]

On this piecewise-uniform mesh, we use a classical upwinded finite difference operator at all points except at the transition point where $x_i=L-\sigma$.
The discrete problem is: Find $U$ such that
\begin{subequations}\label{discrete-problem}
\begin{align}
 L_F^{N,M} U(x_i,t_j) & =f(x_i,t_j), \quad (x_i,t_j) \in \Omega ^N_S;\\
U(0,t_j)& = u(0,t_j) , \ t_j \geq 0, \quad  U(x_i,0)= u(x_i,0) ,\  0 \leq x_i \leq L,
\end{align}
where the fitted finite difference operator $L_F^{N,M}$ is defined as
\begin{align}
 &L_F^{N,M}U(x_i,t_j) \nonumber \\
 &:= \begin{cases}
 L^{N,M}U(x_i,t_j), & \text { if }  x_i \neq L-\sigma, \\
\left(L^{N,M}+a( x_i,t_j)\left( \frac{\rho}{1-e^{-\rho}} -1\right) D^-_x \right)U( x_i,t_j),
&  \text { if }  x_i = L-\sigma,
\end{cases}
\end{align}
\end{subequations}
with  $\rho := H/\ve.$ A standard proof-by-contradiction argument can be used to establish the next result.

\begin{lemma} \label{minprindis}
 Let $S=(p_1, p_2]\times (q_1, q_2]$ be a subdomain of $\Omega$. For any mesh function $Z$, if $Z(x_i,q_1) \geq 0, \, p_1 \leq x_i \leq p_2,  \, Z(p_1, t_j) \geq 0, \, q_1 \leq t_j \leq q_2$ and $L_F^{N,M}Z(x_i,t_j) \geq 0, \ (x_i,t_j)   \in S $,
 then $Z(x_i,t_j) \geq 0, \ (x_i,t_j)  \in \bar S$.
\end{lemma}

The discrete solution can be decomposed into a regular and singular component, $U=V+W$, where the  regular component satisfies
\begin{align*}
L_F^{N,M} V(x_i,t_j) & =0, \quad (x_i,t_j) \in \Omega ^N_S, \\
V(0,t_j)& = u(0,t_j) , \ t_j \geq 0, \quad  V(x_i,0)= u(x_i,0) ,\  0 \leq x_i \leq L,
\end{align*}
and the layer function satisfies
\begin{align*}
L_F^{N,M} W(x_i,t_j) & =f(x_i,t_j), \quad  (x_i,t_j) \in \Omega ^N_S, \\
W(0,t_j) & = 0, \ t_j \geq 0, \qquad  W(x_i,0)= 0,\ 0\leq x_i \leq L.
\end{align*}
\begin{lemma}\label{Lemma4} The discrete boundary layer component satisfies
\[
\vert W(x_i,t_j) \vert \leq C N^{-1}, \  x _i \leq L-\sigma, \ t_j \geq 0.
\]
\end{lemma}
\begin{proof} For all $x_i < L-\sigma $, consider the barrier function
\[
 B_1(x_i,t_j) :=  \frac{C_1}{\ve} \frac{ H}{\alpha}  \sum _{n=1}^i  e^{\frac{ x_n-L}{\ve }} \quad \hbox{with} \quad
D_x^-B_1 (x_i,t_j) =  \frac{C_1}{\alpha \ve} e^{\frac{ x_i-L}{\ve }},
\]
 and the constant $C_1$ is given in~\eqref{flow-towards}. From Lemma~\ref{minprindis},  at all mesh points $x_i <L -\sigma $, outside the layer
\[
L_F^{N,M} B_1(x_i,t_j) \geq a(x_i,t_j) D^-_x  B_1 (x_i,t_j) \geq \vert f(x_i,t_j)\vert .
\]
Hence, for $x_i < L-\sigma$,
\[
W(x_i,t_j) \leq  B_1(x_i,t_j) \leq C\frac{H}{\ve} e^{\frac{ x_i-L}{\ve}}\frac{1-e^{-\frac{L}{\ve}}}{1-e^{-\frac{ H}{\ve }}}
\leq Ce^{\frac{ x_{i+1}-L}{\ve}}\leq Ce^{-\frac{\sigma}{\ve}} \leq CN^{-1}.
\]
At the transition point $x_i=L-\sigma$, use Lemma~\ref{minprindis} and the barrier function
\[
 B_2(L-\sigma ,t_j) = C_1\frac{H}{\ve}  e^{-\frac{ \sigma}{\ve}} \frac{(1-e^{-\rho })}{\rho} +CN^{-1},
\quad  B_2(L-\sigma -H,t_j) = CN^{-1},
\]
to complete the proof.
\end{proof}


\begin{theorem}\label{Result1}
If $U$ is the solution of (\ref{discrete-problem}) and $u$ is the solution of (\ref{main-problem}), (\ref{compat}), (\ref{flow-towards}) then
\[
\Vert \bar U -u \Vert  \leq  C (N^{-1}\ln N + M^{-1}\ln M).
\]
\end{theorem}
\begin{proof} The nodal error is decomposed into two components
\[
u - U =  (v-V)+(w-W).
\]
As the derivatives of $v$ are bounded independently of $\ve$, we  have the truncation error bounds
\begin{align*}
\vert L_F^{N,M}(V-v)(x_i,t_j) \vert & \leq C(N^{-1}+ M^{-1}), \quad x_i \neq L-\sigma \\
\vert L_F^{N,M}(V-v)(L-\sigma ,t_j) \vert & \leq C(N^{-1}+ M^{-1})  + C\frac{\rho }{ 1-e^{-\rho} }.
\end{align*}
Use the discrete barrier function
\[
 B_3(x_i,t_j) =CM^{-1} t_j + CN^{-1}
\begin{cases}
x_i, &  x_i < L-\sigma, \\ x_i+1, & x_i \geq L-\sigma,
 \end{cases}
\]
and the discrete comparison principle  (Lemma~\ref{minprindis}) to deduce the nodal error bound
 \begin{equation}\label{boundV}
\vert (V-v)(x_i,t_j) \vert \leq B_3(x_i,t_j)  \leq C(N^{-1}+ M^{-1}).
\end{equation}
Outside the fine mesh, using (\ref{decomp-bounds}b) and Lemma \ref{Lemma4}, we have that
\[
\vert (W-w)(x_i,t_j) \vert  \leq \vert W(x_i,t_j) \vert +\vert w(x_i,t_j) \vert \leq  C N^{-1}, \qquad x_i \leq L-\sigma.
\]
Within the fine mesh in both space and time, we have  the truncation error bound  for $L-\sigma < x_i < L, \, t_j \leq \tau$,
\[
\vert L_F^{N,M}(W-w)(x_i,t_j) \vert \leq C\frac{N^{-1}\ln N + M^{-1}\ln M}{\ve}e^{\frac{ x_i-L}{\ve}},
\]
where we use $\ve ^{-1} \leq C_2  \ln N$ in the case where  the mesh is piecewise uniform in space.
For $L-\sigma < x_i < L, \tau < t_j \leq T$, using (\ref{extra-decomp}) and (\ref{sharper})
\begin{eqnarray*}
\left \vert \left( D^-_t - \frac{\partial }{\partial t} \right) w  \right \vert  \leq C \ve ^{-1} e^{{-\frac{p\alpha \tau}{\vr}}} e^{\frac{ x_i-L}{\ve}}+ \left \vert \left( D^-_t - \frac{\partial }{\partial t} \right) w_1 \right \vert \leq C\frac{M^{-1}}{\ve} e^{\frac{ x_i-L}{\ve}}.
\end{eqnarray*}
Hence, for $L-\sigma < x_i < L, t_j  >0 $,
\[
\vert L_F^{N,M}(W-w)(x_i,t_j) \vert \leq C\frac{N^{-1}\ln N + M^{-1}\ln M}{\ve}e^{\frac{ x_i-L}{\ve}},
\]
and
$\vert (W-w) (L-\sigma ,t_j) \vert \leq CN^{-1} , (W-w) (x_i ,0) =0$. Then, using
\[
 B_4(x_i,t_j)= C(N^{-1}\ln N + M^{-1}\ln M)e^{\frac{ x_i-L}{\ve}},
\]
as a discrete barrier function in the fine mesh, we can establish for sufficiently large $N$ and $M$
that
\begin{equation}\label{boundW}
\vert (W-w)(x_i,t_j) \vert \leq C(N^{-1}\ln N + M^{-1}\ln M)e^{\frac{ x_i-L}{\ve}};
\end{equation}
where we use the fact that for $0 < z \leq \delta$
\[
\frac{1-e^{-z}}{z} \geq \frac{1-e^{-\delta }}{\delta } \geq \frac{1}{2} \quad \hbox{for sufficiently small} \quad \delta .
\]
Combining all of the bounds (\ref{boundV}) and  (\ref{boundW}), we deduce the nodal error bound
\[
\Vert u - U \Vert _{\Omega ^N_S} \leq C(N^{-1}\ln N + M^{-1}\ln M).
\]
Combine the arguments in \cite[Theorem 3.12]{fhmos} with the interpolation bounds in \cite[Lemma 4.1]{styor4} and the bounds on the derivatives of the components $v,  w$
to extend this bound to a global error bound.
\end{proof}

\begin{remark} If, in addition to the  constraints~\eqref{flow-towards}, we have $f(x,0)=0, \,  x \in [0, L],$ then $w_0 \equiv 0$ in (\ref{extra-decomp}) and the bounds in (\ref{sharper}) apply to $w$. Hence, in this case, one can use a uniform mesh in time and retain the same error bound given in Theorem \ref{Result1}.
 \end{remark}

\section{Transporting a pulse}\label{sec:pulse}

In the second problem class, the initial or boundary condition contains a  layer, which generates an interior layer in the solution of the transport equation.
Consider problem (\ref{main-problem}), (\ref{compat}) with an $\ve$-dependent initial condition
\begin{subequations}\label{pulse}
\begin{equation}
u(x,0) = \phi_1(x) +\phi _2(x;\ve);
\end{equation}
 where $\phi_1(x)$ is a smooth function and $\phi_2$ has a layer in an interior point $d$ with $0 <d < L$ in that
\begin{equation}
\phi _2^{(i)}(0)=0, \ \vert \phi _2^{(i)}(x) \vert \leq C \ve ^{-i} e^{-\frac{\vert x-d\vert}{\ve }}, \ 0 \leq i \leq 2;
\end{equation}
 or if the layer occurs at the endpoint $d=0$ then
\begin{equation}
\vert \phi _2^{(i)}(x) \vert \leq C x^ \ell \ve ^{-(\ell +i)} e^{-\frac{x}{\ve }}, \ 0 \leq i \leq 2;\quad \ell \geq 3.
\end{equation}
\end{subequations}
 The condition on $\ell$ is required to guarantee the regularity of the singular component~\eqref{decomp2} associated with $u$.  In this section the  problem data, $a,b,f $,  are smooth functions that do not depend on the singular perturbation parameter $\ve$.
 The pulse in the initial condition is transported along the characteristic  curve $x=g(t;d,0) $,
 which is the solution of the initial value problem
\begin{equation}\label{charact}
\frac{dx}{dt} = a(x,t), \qquad x(0)=d.
\end{equation}
Note that
\begin{align*}
  \vert \phi _2(x) \vert & \leq CN ^{-1}, && \hbox{if }  \vert x-d \vert \geq \ve \ln N \  \hbox{and } d >0 \qquad \hbox{ or}  \\
  \vert \phi _2(x) \vert & \leq C_q N ^{-(q- \ell )} (N^{-1}\ln N) ^ \ell, && \hbox{if }   x \geq q\ve \ln N \ \hbox{ and } d =0.
\end{align*}
If $\ve \ln N$ is sufficiently large ($\geq L/2)$, then we can solve for $u$ using a classical scheme on a uniform mesh.
Otherwise, we align the mesh along the characteristic curve passing through $(d,0)$.

The solution $u$ can be decomposed into the sum of a regular component $v \in C^2(\bar \Omega)$ and a layer component $w\in C^2(\bar\Omega)$, defined as the solutions of
\begin{subequations}\label{decomp2}
\begin{align}
Lv &= f, \quad (x,t)  \in \Omega; \quad v(0,t)=u(0,t),\ v(x,0)= \phi _1(x), \\
Lw &= 0, \quad (x,t)  \in\Omega;\quad w(0,t)=0,\ w(x,0) = \phi _2(x,\ve).
\end{align}
\end{subequations}
Based on this decomposition, we first  generate a discrete approximation $V$ to $v$ using a uniform rectangular mesh and simple upwinding.
Consider the transformed variables, $\widetilde u(s,t) = u(x,t)$ where
$
s:= x-g(t; d,0)
$
and $g'(t; d,0)=a(g(t; d,0),t)$, then
\[
\widetilde w_t + [\widetilde  a(s,t) -\widetilde  a(0,t)] \widetilde w_ s + \widetilde b \widetilde w =0, \quad \widetilde w(s,0) = \phi _2(s+d).
\]
Note that 
\[
\left \vert \frac{\partial ^{i+j} \widetilde w}{ \partial s^i  \partial t ^j} (s,t) \right \vert \leq C\ve ^{-i} e^{-\frac{\vert s\vert}{\ve }}, \quad 0 \leq i+j \leq 2.
\]
 Throughout the paper we assume that the characteristic curves $x=g(t; d,0)$ can be explicitly determined. To generate a parameter-uniform approximation to $w$, we solve the problem in the transformed variables $(s,t)$. We describe below a numerical scheme  where it is assumed that $a_x$  does not change sign either side of $x=g(t; d,0)$. Otherwise, a more general scheme is required which is described in the Appendix of this paper.

 We use a uniform mesh in time and a piecewise-uniform mesh in space, where the space domain is split into the subdomains
\begin{subequations}\label{transition}
\begin{align}
&[-d,-\sigma _d]\cup [-\sigma _d, \sigma _d] \cup [\sigma _d ,L-d], \nonumber \\
& \hspace{1cm} \sigma _d:= \min \left \{ \frac{d}{2}, \frac{L-d}{2}, \ve \ln N \right \}, \text{ if } d >0; \\
&[0, \sigma _0] \cup [\sigma _0 ,L], \ \sigma _0:= \min \left \{ \frac{L}{2}, ( \ell+1)\ve \ln N \right \}, \text{ if }  d =0.
\end{align}
\end{subequations}
 Then, consider the  following  upwind scheme{\footnote{The upwind finite difference operator is defined to be
\[
b(s_i,t_j)  D^*_ s \widetilde W : = 0.5\Bigl( (b(s_i,t_j) + \vert b(s_i,t_j)\vert ) D^-_ s + (b(s_i,t_j) - \vert b(s_i,t_j)\vert ) D^+_ s \Bigr) \widetilde W.
\]
}}  on the subdomain $\widetilde \Omega:= (-d,L-d)\times (0,T]$
\begin{subequations}\label{scheme-P-singular}
\begin{align}
&D_t^- \widetilde W +[\widetilde  a(s_i,t_j) -\widetilde a(0,t_j)]   D^*_ s \widetilde W+\widetilde b\widetilde W= 0,  \quad (s_i,t_j)  \in \widetilde \Omega; \\
&\widetilde W(-d,t_j) = \widetilde W(L-d,t_j) = 0; \ \widetilde W(s_i,0) = \phi _2(s_i+d).
\end{align}
At $s_i=0$, the stencil collapses to the two point scheme
\begin{equation}
(D_t^- \widetilde W +\widetilde b\widetilde W)(0,t_j) = 0,\ t_j >0,  \ \widetilde W(0,0) = \phi _2(d).
\end{equation}
\end{subequations}
Hence $\widetilde W (0,t_j)$ can be easily determined and the remaining nodal values for the interior layer function, $\widetilde W(s_i,t_j), s_i \in (-d,L-d)\setminus \{ 0\}$, can be determined  separately on either side of $s=0$.
The layer component $w$ is similarly approximated when $d=0$.

We  can generate a global  approximation $\bar W$ to $w$ over the region $ \bar \Omega$
{\footnote{For $x >L, 0< t \leq T$, the data $a,b$ can be smoothly extended so that the problem (\ref{decomp2}b) is well defined over $\bar \Omega$ and the non-negativity of $a,b$ is retained by their extensions.}} using  bilinear interpolation. Then set
\[
\bar W \equiv 0 \text { on } \bar \Omega \backslash\left\{ (x,t), \,  -d+g(t;d,0)\le x \le L, \ 0 \le t \le T \right\}.
\]
We  form $\bar U=\bar V + \bar W$,  where $\bar V$ is the global approximation to $v$ with a classical scheme on a uniform mesh, and we can deduce the global error bound, over $\bar \Omega$,
\[
\Vert u - \bar U \Vert   \leq C N^{-1} (\ln N )^2+ CM^{-1}
\]
for the problem class (\ref{main-problem}), (\ref{compat}), (\ref{pulse}).

 A related problem to  (\ref{main-problem}), (\ref{compat}), (\ref{pulse}) when the pulse occurs in the boundary condition is considered in the Appendix. The numerical approximation to this problem class is used in \S\ref{sec:numerical}.

\begin{remark}
The construction of the numerical schemes of this section  and the Appendix is based on the characteristic curves $g(t;d,0)$. If the location of the characteristic curves need to be estimated, then we can use a  Runge-Kutta method to solve the nonlinear ode $x'(t)=a(x,t), \, x(0) = s$. However, to preserve parameter-uniform convergence then for any  characteristic curves passing through $(x(0),0)$ where $x(0) \in [d-2 \sigma_d ,d+2  \sigma_d ]$ needs to be determined so that $\vert x(t_i)- x^N(t_i) \vert \leq \sigma_d $ for an approximate curve $( x^N(t),t)$.
\end{remark}
\section{\bf Flow away from an attractive force} \label{sec:FlowAway}

In the third problem class to be examined, the forcing term has a layer at the inflow of the domain.
This generates a boundary layer on the left and an interior layer emanating from the initial inflow  boundary point.

Consider problem (\ref{main-problem}), (\ref{compat}) with the additional conditions
\begin{subequations}\label{flow-away-gen}
\begin{align}
&a_t(0,0) =0, \quad
\left \vert \frac{\partial ^{i+j} f}{\partial x^i\partial t^j} (x,t) \right \vert  \leq  C x^{2-i}\ve ^{i-3}  e^{-\frac{x}{\ve}}, \ 0\leq i+j \leq 2; \\
&a \in C^3(\bar \Omega), \quad b(x,t) \equiv 0,\ f(x,t)=f(x), \   \forall (x,t) \in \bar \Omega .
\end{align}
\end{subequations}
In this section the  problem data, $a,b, \psi , \phi $,  are smooth functions that do not depend on the singular perturbation parameter $\ve$.
By Lemma \ref{minprin},
\[
\vert u(x,t) \vert \leq  \frac{C}{\alpha} \left(1-e^{\frac{ -x}{\ve}}\right).
\]
{\bf Notation:}  For each $t\geq 0$, define $w_0,w_1 \in C^2(\bar \Omega)$ as the solutions of
\begin{subequations}\label{w01}
\begin{align}
a(x,t) \frac{\partial w_0}{\partial x}  &= f(x), \quad w_0(\infty ,t) =0,
\\
a(x,t) \frac{\partial w_1}{\partial x} &= - \frac{\partial w_0}{\partial t}, \quad w_1(\infty ,t) =0.
\end{align}
\end{subequations}
 These functions are used below to show that the solution $u$ exhibits a boundary layer near $x=0$. They satisfy
\[
\left \vert \frac{\partial ^{i+j} w_0}{\partial x ^i\partial t ^j}  (x,t) \right \vert \leq C \ve ^{-i}e^{\frac{ -x}{\ve}},
\quad \left \vert \frac{\partial ^{i+j} w_1}{\partial x ^i\partial t ^j}  (x,t) \right \vert \leq C \ve ^{1-i}e^{\frac{ -x}{\ve}}, \quad 0 \leq i+j \leq 2.
\]

We can  decompose the solution into three distinct components
\[
u(x,t) =v(x,t)+ w(x,t) +z(x,t),
\]
 where $w$ contains a boundary layer and $z$ contains an interior layer. They are defined as the solutions of the problems:
\begin{subequations}\label{decomposition-gen}
\begin{align}
Lv  &= 0, \ (x,t)  \in \Omega ; \quad  v(0,t) = u(0,t),\  v(x,0) =  u(x,0), \ v \in C^2(\bar \Omega); \\
Lw  &= f(x), \ (x,t)  \in \Omega ; \quad  w(0,t) = 0,\  w(x,0) =  w_0(x,0)-w_0(0,0); \\
Lz  &= 0, \ (x,t)  \in \Omega;\quad z(0,t) = 0,  z(x,0) =  - w(x,0).
\end{align}
\end{subequations}
Note that the initial condition for the component $w$ is equivalent to
\[
a(x,0) \frac{\partial w}{\partial x}(x,0)=f(x).
\]
In addition, the interior layer component $z$ is given by
\[
z(x,t) =
\begin{cases}
0, &  x \leq g(t;0,0), \\
w_0(0,t) -  w_0(x-g(t),t) , &  x \geq g(t;0,0).
\end{cases}
\]
From (\ref{flow-away-gen}a)  $f(0)=f'(0)=0$ and hence $w(0,0)=w_x(0,0)=w_{xx}(0,0)=0$. By (\ref{compat}) it follows that  $w,z \in C^2(\bar \Omega)$.
A global approximation $\bar Z$ to the solution of   (\ref{decomposition-gen}c) can be generated using the algorithm in \S\ref{sec:pulse}, where $d=0$.
Note that this appoximation is generated on a mesh, defined in the transformed coordinate system.
 We next establish  bounds on the first and second derivatives of $w$.
\begin{lemma} There exists a function $\chi  \in C^2(\bar \Omega)$ such that  the solution of (\ref{decomposition-gen}b)
satisfies the following bounds
\begin{subequations}\label{conj}
\begin{align}
\vert w(x,t) - \chi (x,t)\vert & \leq Ce^{\frac{ -x}{\ve}},  \
 \left \Vert \frac{\partial ^{i+j} \chi }{\partial x ^i\partial t ^j}  \right \Vert  \leq C , i+j \leq 2; \label{conj1}  \\
 \left \vert \frac{\partial ^{i+j} (w -\chi) }{\partial x ^i\partial t ^j}  (x,t) \right \vert & \leq C \ve ^{-i}e^{\frac{ -x}{\ve}}, 0 \leq i+j \leq 2 \label{conj3}.
\end{align}
\end{subequations}
\end{lemma}
\begin{proof}  Consider the functions $w_0$ and $w_1$ defined in~\eqref{w01}.  Define $\chi  \in C^2(\bar \Omega)$ to be
\begin{equation}\label{chi}
\chi (x,t) :=(w-w_0-w_1)(x,t),
\end{equation}
which is the solution of the problem
\begin{subequations} \label{edp:chi}
\begin{align}
 L\chi & = - \frac{\partial w_1}{\partial t}, \quad  \chi (0,t) =- (w_0+w_1)(0,t),
\\
\chi (x,0) & = - w_0(0,0) - w_1(x,0).
\end{align}
\end{subequations}
Observe that,
\begin{align*}
\left \vert \frac{\partial ^{i+j} L\chi (x,t)}{\partial x ^i\partial t ^j}   \right \vert & =
\left \vert \frac{\partial ^{i+j+1} w_1 (x,t)}{\partial x ^i\partial t ^{j+1}}   \right \vert \leq C \ve ^{1-i}e^{\frac{ -x}{\ve}}, \quad 0 \leq i+j \leq 2;
\\
\left \vert \frac{\partial ^{j} \chi}{\partial t ^j}  (0,t) \right \vert &\leq C, \quad 0 \leq j \leq 2; \quad
\left \vert \frac{\partial \chi}{\partial x}  (x,0) \right \vert \leq  C \left \vert \frac{\partial w_1}{\partial x}  (x,0) \right \vert \leq C.
\end{align*}
Note also that
\[
\frac{\partial ^2\chi }{\partial x ^2} (x,0) =- \frac{\partial ^2w_1}{\partial x ^2}  (x,0)   = \left(  \frac{\partial }{\partial x } \left( \frac{1}{a}\right)  \frac{\partial w_0}{\partial t}   +\frac{1}{a} \frac{\partial }{\partial t}\left(  \frac{f}{a}\right) \right) (x,0) .
\]
If $a_t(0,0) = 0$, then
\[
\left \vert \frac{\partial ^2\chi}{\partial x ^2}  (x,0) \right \vert \leq C +C x \vert f(x) \vert \leq C.
\]
Hence, bounding $\chi $ and its derivatives as in Lemma~\ref{Lemma2}, we deduce that
\[
\left \Vert \frac{\partial ^{i+j} \chi }{\partial x ^i\partial t ^j}  \right \Vert \leq C ,\qquad  0 \leq i+j \leq 2.
\]
 Estimates~\eqref{conj3} follow from the definition~\eqref{chi} of $\chi$ and the bounds on the derivatives for $w_0$ and $w_1$.
\end{proof}

Based on these bounds on the derivatives of $w$, we will use a  Shishkin mesh   in  space and a uniform mesh in time, which we shall denote by
 $\Omega ^N_{SU} $.  The space domain is $[0,L]= [0, \sigma ]\cup [\sigma ,L]$ with the transition point as
\[
\sigma := \min \left \{ \frac{L}{2}, \ve\ \ln {N} \right \}.
\]
On this  mesh, we use a classical upwinded finite difference operator.
The discrete problem is: Find $W$ such that
\begin{subequations}\label{discrete-away}
\begin{align}
L^{N,M}W(x_i,t_j) &=  f(x_i), \quad (x_i,t_j) \in \Omega ^N_{SU};\\
 W(0,t_j)&=0,\ t_j \geq 0; \quad a(x_i,0)D_x^- W(x_i,0)=f(x_i),\ x_i >0.
\end{align}
 This new scheme also satisfies a discrete comparison principle.
\end {subequations}
\begin{theorem}\label{Main}
If $W$ is the solution of (\ref{discrete-away}) and $w$ is the solution of (\ref{decomposition-gen}b),  then
\[
\Vert \bar W -w \Vert  \leq CN^{-1} (\ln N)^2 +CM^{-1}.
\]
\end{theorem}
\begin{proof}
In analogous fashion to the decomposition of $w$, the discrete function $W$ can be decomposed into several  components.  Denote by $W_0,W_1, X$ the discrete counterpart of $w_0,w_1$ and $\chi$ given in~\eqref{w01} and~\eqref{edp:chi}. These functions are defined as the solutions of the following discrete problems.
\begin{align*}
& a(x_i,t_j) D^-_xW_0(x_i,t_j) =f(x_i), \quad W_0(L,t_j) =0;  \\
& a(x_i,t_j) D^-_xW_1(x_i,t_j) =-D^-_tW_0(x_i,t_j), \quad W_1(L,t_j) =0; \\
& L^{N,M}X (x_i,t_j)=-D^-_tW_1(x_i,t_j), \\
& X(0,t_j)=-(W_0+W_1)(0,t_j), \  X(x_i,0)=-W_0(0,0)-W_1(x_i,0) .
\end{align*}
By noting that $W(x_i,0)= W_0(x_i,0)-W_0(0,0)$, we see that
\[
W=W_0+W_1+X.
\]
For any mesh function $Z$, if
\begin{subequations}\label{closed-form}
\begin{align}
 D_x^-Z(x_i)&=g(x_i), \, x_i < x_N, \quad Z(x_N)=g(x_N), \quad  \hbox{then} \\
  Z(x_i) &=g(x_N) - \sum _{n=i+1}^N h_ng(x_n).
\end{align}
\end {subequations}
Observe also  that
\begin{subequations}\label{star}
\begin{align}
\frac{\partial }{\partial x} \Bigl (\frac{\partial w_0}{\partial t}  \Bigr)(x,t)  &= f(x) \frac{\partial a^{-1}}{\partial t} ,\quad  \frac{\partial w_0}{\partial t}  (\infty ,t) =0 \quad \hbox{and} \\
 D_x^-\bigl( D_t^- W_0) (x_i,t_j) & = f(x_i) D^-_t(a^{-1}(x_i,t_j)), \ D_t^- W_0(L,t_j)=0.
\end{align}
\end {subequations}
Hence, using (\ref{closed-form}), we deduce that
\begin{eqnarray*}
\vert W_0(x_i,t_j)  \vert , \ \vert D^-_tW_0(x_i,t_j)  \vert, \ \vert W_1(x_i,t_j)  \vert , \ \vert D^-_tW_1(x_i,t_j)  \vert \leq Ce^{-\frac{x_i}{\ve}}.
\end{eqnarray*}
Then  for $W_n, \, n=0,1$ , outside the fine mesh, where $x_i \geq \sigma$,  we have
\begin{align*}
\vert (W_n - w_n)(x_i,t_j) \vert & \leq \vert W_n (x_i,t_j) \vert +\vert  w_n(x_i,t_j) \vert \leq C N^{-1};\\
\vert D^-_t(W_n - w_n)(x_i,t_j) \vert & \leq  \vert D^-_tW_n (x_i,t_j)\vert +C \left \Vert  \frac{\partial w_n}{\partial  t} (x_i,t) \right \Vert \leq C N^{-1},
\end{align*}
and within the fine mesh, where $x _i < \sigma$,
\begin{align*}
\vert D_x^- (W_0 -w_0)(x_i,t_j)  \vert & \leq C \left \vert D_x^-w_0 - \frac{\partial w_0 }{\partial x} \right \vert \leq C\frac{h}{\ve ^2} e^{-\frac{x_i}{\ve}}, \\
\vert (W_0 -w_0)(x_i,t_j)  \vert & \leq C \sum _{n=i+1}^{N/2} \frac{h^2}{\ve ^2} e^{-\frac{x_i}{\ve}} +CN^{-1} \leq CN^{-1} \ln N.
\end{align*}
Using  (\ref{star}) and repeating the argument we get that for $x_i < \sigma$
\[
\vert D^-_t (W_0 -w_0)(x_i,t_j)  \vert \leq CN^{-1} \ln N +CM^{-1},
\]
and then for $x_i < \sigma$
\begin{align*}
\vert (W_1 -w_1)(x_i,t_j)  \vert & \leq  CN^{-1} \ln N +CM^{-1}, \\
\vert D^-_t (W_1 -w_1)(x_i,t_j)  \vert & \leq CN^{-1} \ln N +CM^{-1}.
\end{align*}
From all of these bounds,  we can then deduce that
\[
\vert (X -\chi )(x_i,t_j)  \vert \leq  CN^{-1} \ln N +CM^{-1}.
\]
Combine all of these bounds together to establish the nodal error bound
\[
\vert (W - w)(x_i,t_j) \vert \leq CN^{-1} (\ln N)^2 +CM^{-1}, \quad (x_i,t_j) \in \Omega ^N_{SU}.
\]
This can be extended to a global error bound as in Theorem \ref{Result1}.
\end{proof}
\begin{theorem}\label{Result3}
If $W$ is the solution of (\ref{discrete-away}), $V$ and $Z$ are the discrete approximations to $v,z$
and $\bar U=\bar V+\bar W+\bar Z$,   then
\[
\Vert \bar U -u \Vert  \leq  CN^{-1} (\ln N)^2+ CM^{-1}.
\]
where $u$ is the solution of (\ref{main-problem}), (\ref{compat}), (\ref{flow-away-gen}).
\end{theorem}

\begin{remark} In the case where $b(x,t) \neq 0$ and $f(x,t)$, we define $w_0$ as the solution of
\[
a(x,t) \frac{\partial w_0}{\partial x}  +b(x,t) w_0= f(x,t), \quad w_0(\infty ,t) =0.
\]
 Then, the numerical method is given by
\begin{align*}
&L^{N,M}W(x_i,t_j)=f(x_i,t_j), \quad  W(0,t_j)=0, \ t_j \geq 0 \\
&(aD^-_x+bI)W(x_i,0)=f(x_i,0), \ x _i >0.
\end{align*}
\end{remark}

\section{Numerical results} \label{sec:numerical}

In this final section we examine a problem which is motivated by a mathematical model of fluid-particle interaction in particle-laden flows  \cite{baranger,Lagoutiere,astro}.
The numerical solution of the problem requires the use of the three algorithms described in \S\ref{sec:FlowTowards}, \ref{sec:pulse} and \ref{sec:FlowAway}.  Let us consider two-phase  fluid flow composed of  a continuously connected phase (e.g. a gas, the carrier phase) and a dispersed phase (e.g. small particles) in a one-dimensional duct.  Assuming all the particles move at the same constant velocity $u_p$, the flow is incompressible, the pressure  gradients and the viscous loses are negligible \cite{baranger}, we have the following model to determine  the velocity of the fluid $u$,  the temperature of the fluid $T$,  the cumulative density of the particles $N_p$ in the $x$ direction  and the temperature of the particles $T_p$:
{\footnote{ The model parameters in the system correspond to a  drag coefficient $\lambda$, a convection coefficient $h$, the surface of the  particle $S_p$, the fluid density $\rho$, the specific heat capacity of the fluid $C_p $,  the thermal diffusivity $k$, the thermal  emissivity $\epsilon$, the Stefan-Boltzmann constant $\varsigma$ and the heat source temperature $T_{\infty}$. }}
\begin{align*}
\frac{\partial u}{\partial t}+u \frac{\partial u}{\partial x} &= -\lambda  S_p(u-u_p) \frac{\partial N_p}{\partial x}; \\
\frac{\partial T}{\partial t}+u \frac{\partial T}{\partial x} &=  r S_p(T_p-T) \frac{\partial N_p}{\partial x}+ \frac{1}{\rho C_p } \frac{\partial ^2 ( kT) }{\partial x^2}, \quad r := \frac{h}{\rho C_p };  \\
\frac{\partial N_p}{\partial t}+u_p \frac{\partial N_p}{\partial x} &= 0;\\
\frac{\partial T_p}{\partial t}+u_p\frac{\partial T_p}{\partial x} &= -r  S_p(T_p-T)+ F, 
\end{align*}
 where $F$ may be a function of the problem variables and other physical parameters and denotes a heat source term. For example,  if the particles are heated by radiation, then $F(\epsilon,\varsigma,S_p,T_\infty,T_p)=\epsilon \varsigma S_p(T_\infty^4-T_p^4)$.

Assuming that the reference frame moves with the particles (i.e. $x \rightarrow x+u_pt$), that the thermal diffusivity is small and that $T_p \gg T$, so that $T_p-T$ can be assumed constant, then the equation for the   fluid temperature  is of the general form:  Find $T(x,t)$ such that
\begin{subequations}\label{Example}
\begin{align}
&T_t+\omega _0 (x,t)  T_x = \beta z_x, \quad (x,t) \in (0,L] \times (0, T_f];\\
&T(0,t) = T_0, \ t \geq 0, \  T(x,0)= \phi (x):= T_0 +A_0e^{-\frac{(x-d_0)^2}{\mu}}, \, 0 \leq x \leq L; \\
&z(x,t) = A_1 \tanh \left( \frac{x-d_1}{\ve } \right).
\end{align}
 Around  $x=d_0$, the temperature is higher at the initial time. There is also a concentration of particles located at $x=d_1 > d_0$ which will exchange energy with the fluid, increasing its temperature. Note that the parameter $\ve$ controls the thickness of the layer of particles. The extreme case of $\ve$ very small corresponds to a single layer of particles.  As the numerical methods designed in this paper are parameter-uniform, the numerical approximations will converge to the true solution irrespective of how small $\ve>0$ is and their global pointwise accuracy only depends on the number of mesh elements used in the computations.

In the case of the fluid temperature, if the velocity $\omega _0$ is constant we can write out the exact solution explicitly
as
\[
T(x,t) =
\begin{cases}
\phi (x-\omega_0 t) + \frac{\beta A_1}{{ \omega_0}} \left(\tanh  \left( \frac{x-d_1}{\ve } \right) -\tanh  \left( \frac{x-{ \omega_0} t-d_1}{\ve } \right) \right) , & x \geq{ \omega_0} t, \\
T_0+ \frac{\beta A_1}{ \omega_0} \left(\tanh  \left( \frac{x-d_1}{\ve } \right) +\tanh  \left( \frac{d_1}{\ve } \right) \right), & x \leq { \omega_0} t.
\end{cases}
\]
Here we shall take a variable  fluid velocity for the fluid temperature as
\begin{equation}
\omega _0 (x,t) =2- \frac{x}{L} \ge  1.
\end{equation}
In our numerical experiments we take the sample parameter values of
\begin{align}
L=10, \ T_f=5, \ \beta =1, \ A_0=50, \ A_1=10, \\
\mu = \vr/4, \ T_0=300, \ d_0=2, \ d_1=5.
\end{align}
\end{subequations}
As the problem (\ref{Example}) is linear, we can split the solution as follows.
\[
T(x,t) = T_0+P(x,t) + R(x,t),
\]
where $P$ solves the problem
\begin{subequations}\label{component -P}
\begin{align}
&P_t+\omega _0  P_x = 0, \ (x,t) \in (0,L] \times (0, T_f];\\
&P(0,t) = 0, \ t \geq 0, \ P(x,0)= \phi (x) -T_0, \ 0 \leq x \leq L.
\end{align}
\end{subequations}
The component $P$ corresponds to the initial pulse at $x=d_0$ transported in time.
To solve for $P$ numerically, we  align the mesh to the characteristic
\[
\frac{dx}{dt}=\omega _0(x), \ x(0)=d_0,
\] and follow the algorithm in \S\ref{sec:pulse}.  Note that the regular component of $P$ is a constant function which takes the value of 300; and so it is only necessary to numerically approximate the singular component of $P$ with the scheme~\eqref{scheme-P-singular}. The transition point is of the form
\[
\min \left \{\frac{d_0}{2}, \frac{d_1}{2},  \sqrt{\mu} \ln N \right \},
\]
either side of $s=0$, where $s=x- g(t;d_0,0)$.
The second  temperature component $R$ (coming from the presence of the particles at $x=d_1$) is the solution of
\begin{subequations}\label{component -R}
\begin{align}
&R_t+\omega _0  R_x = \beta z_x, \quad (x,t) \in (0,L] \times (0, T_f];\\
&R(0,t) = 0, \ t \geq 0, \qquad  R(x,0)= 0, \ x \geq 0.
\end{align}
\end{subequations}
An approximation to $R$ is generated separately on $\Omega ^-:= (0,d_1] \times (0, T_f]$ and $\Omega ^+:=(d_1,L] \times (0, T_f]$.
As $z_x(d_1,0) \neq 0$, we observe that $R \not\in C^1(\bar \Omega ^+)$.
We use the fitted operator algorithm on the Shishkin mesh $\Omega _S^N$ (described in \S\ref{sec:FlowTowards}) to generate an approximation $R^N$ to $R$ on the domain $\bar \Omega ^-$.
Hence, $R^N(d_1,t_j) = \Psi (t_j)$ will be determined.
Finally, over $\Omega ^+$ we further subdivide $R=S+I$ where $S \in C^0(\bar \Omega ^+)\setminus  C^1(\bar \Omega ^+)$
\begin{subequations}\label{component -S}
\begin{align}
&S_t+\omega _0  S_x = \beta z_x, \ (x,t) \in \Omega ^+; \\
&S(d_1,t) = 0, \ t \geq 0, \quad  S(x,0)=\beta \int _{s=d_1}^x\frac{z_x(s)}{\omega _0(s) }  ds  , \ d_1 \leq x \leq L,
\end{align}
\end{subequations}
and $I \in C^1(\bar \Omega ^+)$ is such that
\begin{subequations}\label{component -I}
\begin{align}
&I_t+\omega _0  I_x = 0, \quad (x,t) \in \Omega ^+:= (d_1,L] \times (0, T_f], \\
&I(d_1,t) = R(d_1,t), \ t \geq 0, \qquad  I(x,0)= -S(x,0),\ d_1 \leq x \leq L.
\end{align}
\end{subequations}

To approximate $S$ over $\Omega ^+$ we use  an upwind operator on the Shishkin mesh $\Omega _{SU}^N$  and follow the procedure in \S\ref{sec:FlowAway}.

To approximate  the final component $I$ we  use a transformation as in \S\ref{sec:pulse} either side of the characteristic curve $x=g(t;d_1,0)$, where $I=0$ at all points along this curve. Consider the following two subdomains
\[
\Omega ^+_L: = \left \{(x,t) \in \Omega ^+ \ | \  x < g(t;d_1,0) \right \}, \quad \Omega ^+_R: = \left \{(x,t) \in \Omega ^+ \ |  \ x >g(t;d_1,0) \right \}
\]
and the associated subproblems
\begin{align*}
& I_t+\omega _0  I_x = 0, \ (x,t) \in  \Omega ^+_L \cup  \Omega ^+_R, \\
& I(d_1,t) = 0, \ t \geq 0, \qquad  I(x,0)= -S(x,0),\ 0\leq x -d_1 \\
 & I(x,0) = 0, \ x -d_1\geq 0, \qquad   I(0,t)= \Psi (t) := R(d_1,t),\ 0\leq t .
\end{align*}
From Lemma \ref{Lemma2}, we note that $R(d_1,t) \leq C e^{-\frac{\kappa t}{\ve}}, \omega _0 (x,t) > \kappa$.
An approximation to $I$ over the subdomain $\Omega ^+_R$ is generated using the transformtion $s=x-g(t;d_1,0)$ and  over the subdomain $\Omega ^+_L$ using the transformtion $\tau =t-g^{-1}(x;d_1,0))$. The numerical method described in \S\ref{sec:pulse} is then applied, with
\[
\sigma _0 = \min \left\{\frac{L}{2}, 2\ve \ln N \right\}  \ \text{ and } \tau _0 = \min \left \{\frac{T}{2}, 2\ve \ln M \right\}.
\]
Then we form $\bar R^N = \bar S^N+\bar I^N$ over $\Omega ^+$.

 In the diagram of Figure~\ref{fig:Diagram} we identify the  components $P,R,S$ and $I$ defined, respectively, in the subdomains $\bar \Omega$, $\bar \Omega^-$, $\bar \Omega^+$ and $\bar \Omega^+_L\cup\bar \Omega^+_R$. All these components are used in our algorithm to approximate the fluid temperature $T$, which is the solution of (\ref{Example}).

\begin{figure}[h!]
\centering
\resizebox{\linewidth}{!}{
		\includegraphics[scale=1, angle=0]{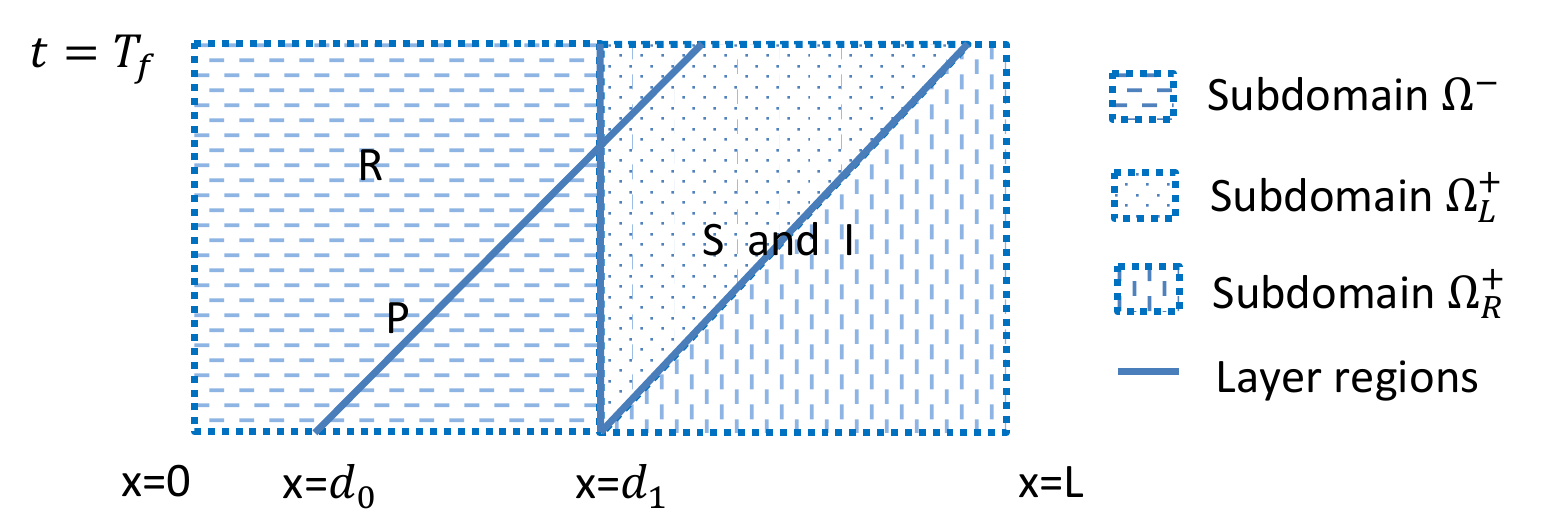}
}
	\caption{Layers regions, subdomains and components of the solution of Example~\eqref{Example}.}
	\label{fig:Diagram}
 \end{figure}

 In Tables~\ref{table:P}-\ref{table:IR} we give the maximum two-mesh global differences and the orders of convergence~\cite{fhmos} associated with the components $P$ (in $\bar \Omega$), $R$ (in $\bar \Omega^-$), $S$ (in $\bar \Omega^+$) and $I$ (in $\bar \Omega^+_L$ and $\bar \Omega^+_R$). The numerical results are only shown for the values of $\vr=2^0,2^{-2},2^{-4},2^{-6}$ and $2^{-30}$, but the uniform two-mesh global differences $D^{N,M}$ and the uniform orders of convergence $P^{N,M}$ are given in the last row of each table taking $\vr=2^{0},2^{-2},\ldots,2^{-30}$. These numerical results indicate that the numerical scheme converges globally and uniformly with almost first order for all the components in agreement with  the theoretical error bounds established in \S 3-5.

\begin{table}[h]
\caption{Maximum and uniform two-mesh global differences and their corresponding orders of convergence for the component $P$ (\ref{component -P}). The numerical results have been obtained with the  scheme~\eqref{scheme-P-singular}}
\begin{center}{\tiny \label{table:P}
\begin{tabular}{|c||c|c|c|c|c|c|c|}
 \hline  & N=M=32 & N=M=64 & N=M=128 & N=M=256 & N=M=512 & N=M=1024& N=M=2048 \\
\hline \hline
 $\vr=2^{0}$
&2.424E+00 &1.365E+00 &7.220E-01 &3.731E-01 &1.903E-01 &9.615E-02 &4.833E-02 \\
&0.829&0.919&0.952&0.971&0.985&0.992&
\\ 
$\vr=2^{-2}$
&3.375E+00 &2.390E+00 &1.343E+00 &7.115E-01 &3.674E-01 &1.869E-01 &9.430E-02 \\
&0.498&0.831&0.917&0.954&0.975&0.987&
\\ 
$\vr=2^{-4}$
&3.030E+00 &2.052E+00 &1.410E+00 &8.574E-01 &5.024E-01 &2.881E-01 &1.685E-01 \\
&0.562&0.542&0.717&0.771&0.802&0.773&
\\ 
$\vr=2^{-6}$
&3.032E+00 &2.054E+00 &1.412E+00 &8.591E-01 &5.036E-01 &2.888E-01 &1.616E-01 \\
&0.562&0.541&0.717&0.771&0.802&0.837&
\\ 
$\vr=2^{-8}$
&3.032E+00 &2.055E+00 &1.412E+00 &8.595E-01 &5.039E-01 &2.890E-01 &1.618E-01 \\
&0.561&0.541&0.716&0.770&0.802&0.837&
\\ 
$\vdots$
&$\vdots$ &$\vdots$ &$\vdots$ &$\vdots$ &$\vdots$ &$\vdots$ &$\vdots$
\\ 
$\vr=2^{-30}$
&3.033E+00 &2.055E+00 &1.412E+00 &8.597E-01 &5.040E-01 &2.891E-01 &1.618E-01 \\
&0.561&0.541&0.716&0.770&0.802&0.837&
\\ \hline $D^{N,M}$
&3.375E+00 &2.390E+00 &1.412E+00 &9.373E-01 &5.111E-01 &2.891E-01 &1.685E-01 \\
$P^{N,M}$ &0.498&0.759&0.592&0.875&0.822&0.778&\\ \hline \hline
\end{tabular}}
\end{center}
\end{table}

\begin{table}[h]
\caption{Maximum and uniform two-mesh global differences and their corresponding orders of convergence for the component $R$ (\ref{component -R}) in the subdomain $\bar \Omega^-$. The numerical results have been obtained with the  { scheme~\eqref{discrete-problem}} }
\begin{center}{\tiny \label{table:R}
\begin{tabular}{|c||c|c|c|c|c|c|c|}
 \hline  & N=M=32 & N=M=64 & N=M=128 & N=M=256 & N=M=512 & N=M=1024& N=M=2048 \\
\hline \hline
 $\vr=2^{0}$
 &2.597E-01 &1.356E-01 &7.431E-02 &3.916E-02 &2.016E-02 &1.024E-02 &5.159E-03 \\
&0.938&0.867&0.924&0.958&0.978&0.989&
\\ 
$\vr=2^{-2}$
&2.928E-01 &1.812E-01 &1.111E-01 &6.804E-02 &3.952E-02 &2.272E-02 &1.276E-02 \\
&0.692&0.706&0.707&0.784&0.799&0.833&
\\ 
$\vr=2^{-4}$
&2.962E-01 &1.824E-01 &1.107E-01 &6.631E-02 &3.988E-02 &2.262E-02 &1.270E-02 \\
&0.700&0.720&0.740&0.734&0.818&0.833&
\\ 
$\vr=2^{-6}$
&2.960E-01 &1.826E-01 &1.089E-01 &6.774E-02 &3.932E-02 &2.277E-02 &1.268E-02 \\
&0.697&0.746&0.685&0.785&0.788&0.844&
\\ 
$\vr=2^{-8}$
&2.960E-01 &1.826E-01 &1.106E-01 &6.624E-02 &3.983E-02 &2.259E-02 &1.268E-02 \\
&0.697&0.723&0.740&0.734&0.818&0.833&
\\ 
$\vdots$
&$\vdots$ &$\vdots$ &$\vdots$ &$\vdots$ &$\vdots$ &$\vdots$ &$\vdots$
\\ 
$\vr=2^{-30}$
&2.959E-01 &1.828E-01 &1.107E-01 &6.777E-02 &3.874E-02 &2.192E-02 &1.282E-02 \\
&0.695&0.724&0.707&0.807&0.822&0.773&
\\ \hline $D^{N,M}$
&2.962E-01 &1.828E-01 &1.226E-01 &6.884E-02 &3.995E-02 &2.323E-02 &1.381E-02 \\
$P^{N,M}$ &0.697&0.576&0.832&0.785&0.782&0.751&\\ \hline \hline
\end{tabular}}
\end{center}
\end{table}

\begin{table}[h]
\caption{Maximum and uniform two-mesh global differences and their corresponding orders of convergence for the component $S$  (\ref{component -S}) in the subdomain $\bar \Omega^+$. The numerical results have been obtained with the { scheme~\eqref{discrete-away}}}
\begin{center}{\tiny \label{table:S}
\begin{tabular}{|c||c|c|c|c|c|c|c|}
 \hline  & N=M=32 & N=M=64 & N=M=128 & N=M=256 & N=M=512 & N=M=1024& N=M=2048 \\
\hline \hline
 $\vr=2^{0}$
 &2.610E-01 &1.303E-01 &6.513E-02 &3.255E-02 &1.627E-02 &8.136E-03 &4.068E-03 \\
&1.002&1.001&1.000&1.000&1.000&1.000&
\\ 
$\vr=2^{-2}$
&2.950E-01 &1.815E-01 &1.085E-01 &6.320E-02 &3.611E-02 &2.031E-02 &1.128E-02 \\
&0.701&0.743&0.779&0.808&0.830&0.848&
\\ 
$\vr=2^{-4}$
&2.967E-01 &1.827E-01 &1.088E-01 &6.326E-02 &3.611E-02 &2.031E-02 &1.128E-02 \\
&0.699&0.748&0.782&0.809&0.830&0.848&
\\ 
$\vr=2^{-6}$
&2.964E-01 &1.826E-01 &1.089E-01 &6.333E-02 &3.613E-02 &2.031E-02 &1.128E-02 \\
&0.699&0.746&0.782&0.810&0.831&0.848&
\\ 
$\vdots$
&$\vdots$ &$\vdots$ &$\vdots$ &$\vdots$ &$\vdots$ &$\vdots$ &$\vdots$
\\ 
$\vr=2^{-30}$
&2.964E-01 &1.826E-01 &1.089E-01 &6.332E-02 &3.614E-02 &2.032E-02 &1.128E-02 \\
&0.699&0.746&0.782&0.809&0.831&0.848&
\\ \hline $D^{N,M}$
&2.970E-01 &1.827E-01 &1.226E-01 &6.513E-02 &3.614E-02 &2.032E-02 &1.128E-02 \\
$P^{N,M}$ &0.701&0.576&0.912&0.850&0.831&0.848&\\ \hline \hline
\end{tabular}}
\end{center}
\end{table}

\begin{table}[h]
\caption{Maximum and uniform two-mesh global differences and their corresponding orders of convergence for the component $I$  (\ref{component -I}) in the subdomain $\bar \Omega^+_L$. The numerical results have been obtained with the { scheme~\eqref{scheme-P-singular-related}}}
\begin{center}{\tiny \label{table:IL}
\begin{tabular}{|c||c|c|c|c|c|c|c|}
 \hline  & N=M=32 & N=M=64 & N=M=128 & N=M=256 & N=M=512 & N=M=1024& N=M=2048 \\
\hline \hline
 $\vr=2^{0}$
 &2.597E-01 &1.356E-01 &7.431E-02 &3.916E-02 &2.016E-02 &1.024E-02 &5.159E-03 \\
&0.938&0.867&0.924&0.958&0.978&0.989&
\\ 
$\vr=2^{-2}$
&2.928E-01 &1.812E-01 &1.111E-01 &6.804E-02 &4.003E-02 &2.289E-02 &1.281E-02 \\
&0.692&0.706&0.707&0.765&0.806&0.838&
\\ 
$\vr=2^{-4}$
&2.962E-01 &1.824E-01 &1.107E-01 &6.780E-02 &3.988E-02 &2.279E-02 &1.275E-02 \\
&0.700&0.720&0.707&0.766&0.807&0.838&
\\ 
$\vr=2^{-6}$
&2.960E-01 &1.826E-01 &1.106E-01 &6.774E-02 &3.984E-02 &2.277E-02 &1.274E-02 \\
&0.697&0.723&0.708&0.766&0.807&0.838&
\\ 
$\vdots$
&$\vdots$ &$\vdots$ &$\vdots$ &$\vdots$ &$\vdots$ &$\vdots$ &$\vdots$
\\ 
$\vr=2^{-30}$
&2.960E-01 &1.826E-01 &1.106E-01 &6.772E-02 &3.983E-02 &2.276E-02 &1.273E-02 \\
&0.697&0.723&0.708&0.766&0.807&0.838&
\\ \hline $D^{N,M}$
&2.962E-01 &1.826E-01 &1.226E-01 &6.884E-02 &4.003E-02 &2.289E-02 &1.281E-02 \\
$P^{N,M}$ &0.698&0.575&0.832&0.782&0.806&0.838&\\ \hline \hline
\end{tabular}}
\end{center}
\end{table}

\begin{table}[h]
\caption{Maximum and uniform two-mesh global differences and their corresponding orders of convergence for the component $I$  (\ref{component -I}) in the subdomain $\bar \Omega^+_R$. The numerical results have been obtained with the { scheme~\eqref{scheme-P-singular}}}
\begin{center}{\tiny \label{table:IR}
\begin{tabular}{|c||c|c|c|c|c|c|c|}
 \hline  & N=M=32 & N=M=64 & N=M=128 & N=M=256 & N=M=512 & N=M=1024& N=M=2048 \\
\hline \hline
 $\vr=2^{0}$
 &3.133E-01 &1.600E-01 &8.101E-02 &4.081E-02 &2.049E-02 &1.026E-02 &5.138E-03 \\
&0.969&0.982&0.989&0.994&0.997&0.998&
\\ 
$\vr=2^{-2}$
&3.488E-01 &2.185E-01 &1.319E-01 &7.773E-02 &4.448E-02 &2.509E-02 &1.395E-02 \\
&0.675&0.728&0.763&0.805&0.826&0.847&
\\ 
$\vr=2^{-4}$
&3.492E-01 &2.190E-01 &1.323E-01 &7.807E-02 &4.469E-02 &2.520E-02 &1.401E-02 \\
&0.673&0.727&0.761&0.805&0.826&0.847&
\\ 
$\vr=2^{-6}$
&3.494E-01 &2.192E-01 &1.325E-01 &7.819E-02 &4.476E-02 &2.525E-02 &1.403E-02 \\
&0.672&0.726&0.761&0.805&0.826&0.847&
\\ 
$\vdots$
&$\vdots$ &$\vdots$ &$\vdots$ &$\vdots$ &$\vdots$ &$\vdots$ &$\vdots$
\\ 
$\vr=2^{-30}$
&3.494E-01 &2.193E-01 &1.326E-01 &7.823E-02 &4.479E-02 &2.526E-02 &1.404E-02 \\
&0.672&0.726&0.761&0.805&0.826&0.847&
\\ \hline $D^{N,M}$
&3.494E-01 &2.193E-01 &1.480E-01 &7.992E-02 &4.479E-02 &2.526E-02 &1.404E-02 \\
$P^{N,M}$ &0.672&0.567&0.889&0.835&0.826&0.847&\\ \hline \hline
\end{tabular}}
\end{center}
\end{table}

 The numerical solution for Example~\eqref{Example} with $\vr =2^{-8}$ is displayed in Figure~\ref{fig:Example} after patching the approximations computed for each one of the components of  $T$. In this figure it is observed that the pulse in the initial condition is transported along the characteristic emanating from the point (2,0) and it merges with the interior layer located at $x=5$.   This interior layer is produced by the heat exchange from the particles to the gaseous phase. In addition, another interior layer emanating from the point (5,0) is observed.  A physical explanation for this is that the initial condition for the temperature is not under equilibrium at $x=5$ (i.e. there is no temperature gradient at (5,0) that satisfies a steady solution of Equation \eqref{Example}). As the particles start heating the gaseous phase at the initial time, the latter increases the temperature while being transported downstream, as observed in the figure in the form of an interior layer emanating from (5,0).

\begin{figure}[h!]
\centering
\resizebox{0.8\linewidth}{!}{
		\includegraphics[scale=1, angle=0]{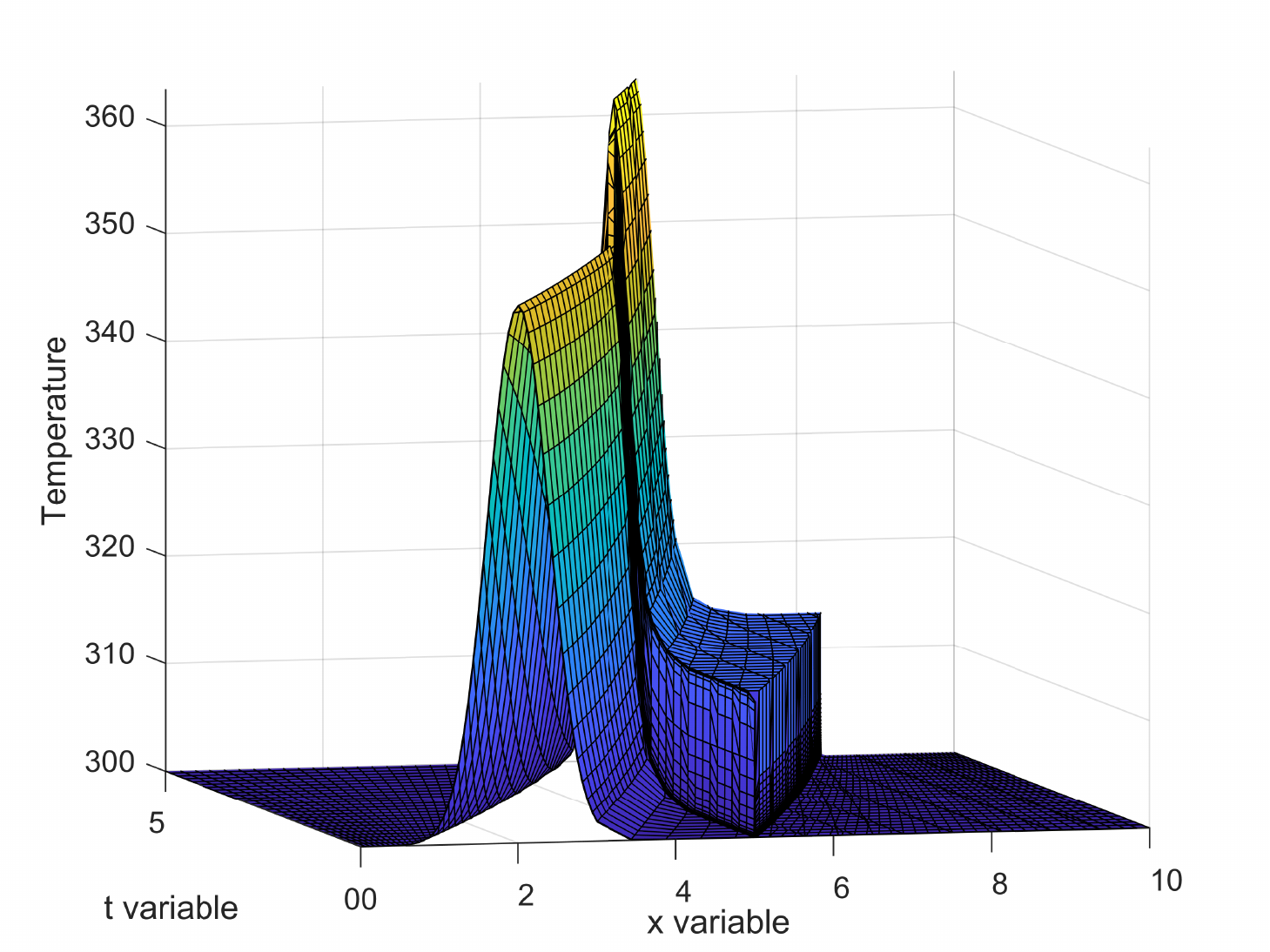}
}
	\caption{ Example~\eqref{Example} with $\vr =2^{-8}$: Computed solution with $N=M=64$.}
	\label{fig:Example}
 \end{figure}

\section*{Conclusions}

Three classes of first order singularly perturbed transport equations are examined. Boundary and interior layers can arise in the solutions of these problems.
Numerical methods are constructed and analysed for these problems  incorporating piecewise-uniform Shishkin meshes, which  insert a significant proportion of the mesh points into the layer regions. In this way, the numerical methods produce parameter-uniform numerical approximations, whose global pointwise accuracy is guaranteed for all possible values of the singular perturbation parameter. A test problem, motivated from modelling fluid-particle interaction, is used to illustrate the performance of these numerical methods.

\section*{ Appendix: Further considerations when approximating a pulse problem} \label{sec:appendix-pulse}

 In \S\ref{sec:pulse} a numerical scheme is described where it is assumed that the convective coefficient $a_x$ does not change sign and the pulse occurs at the initial condition. In this appendix we describe an algorithm to approximate the problem transporting a pulse when either $a_x$ does change sign or the pulse occurs in the boundary condition.
\subsection*{ General case where $a_x$  may change sign}

 If, for some $t=t_j$, the derivative $a_x$  does change sign for $x<g(t; d,0)$ (or  $x>g(t; d,0)$), then  the scheme~\eqref{scheme-P-singular} should not be used to generate a numerical approximation to the singular component~$w$. Observe that in this case the matrix associated with the scheme~\eqref{scheme-P-singular} is a reducible matrix.

 In the case that $a_x$  does change sign, for $s <0$ (or $s>0$ ),
we would simply solve in the fine mesh region $[-\sigma _d, 0]$ (or $[0, \sigma _d]$ ) with the boundary value $\widetilde W(-\sigma _d,t_j) =0$ (or $\widetilde W(\sigma _d,t_j) =0$). That is, we would solve
\begin{align*}
&D_t^- \widetilde W +[\widetilde  a(s_i,t_j) -\widetilde a(0,t_j)]   D^*_ s \widetilde W+\widetilde b\widetilde W= 0, \quad (s_i,t_j)  \in (-\sigma _d,L-d)\times (0,T];\\
&\widetilde W(-\sigma _d,t_j) = \widetilde W(L-d,t_j) = 0; \ t_j>0 \ \widetilde W(s_i,0) = \phi _2(s_i+d) .
\end{align*}
Similarly to the case that $a_x$ does not change sign, a global  approximation $\bar W$ to $w$ is generated over the region $ \bar \Omega$
and we set
\[
\bar W \equiv 0 \text { on } \bar \Omega \backslash\left\{ (x,t), \,  -\sigma_d+g(t;d,0)\le x \le L, \ 0 \le t \le T \right\}.
\]

\subsection*{A pulse in the boundary condition}

Consider problem (\ref{main-problem}), (\ref{compat}) with an $\ve$-dependent boundary condition
\begin{subequations}\label{pulse2}
\begin{align}
&u(0,t) = \psi_1(t) +\psi _2(t;\ve); \\
& \psi _2^{(j)}(0) =0 , \  \vert \psi _2^{(j)}(t) \vert \leq C \ve ^{-j}e^{-\frac{\vert t-d\vert}{\ve }};\ 0 \leq j \leq 2,  \quad  \text{if } 0 <d < T,\\
& \vert \psi _2^{(j)}(t) \vert \leq C t^{ \ell}\ve ^{{ -(\ell +j)}}e^{-\frac{ t}{\ve }};\ 0 \leq j \leq 2;\quad  \ell \geq 3, \quad  \text{if }d=0.
\end{align}
\end{subequations}
 The pulse in the boundary condition is transported along the  characteristic  curve
$x=g_1(t;0,d) $,
 which is the solution of the initial value problem
\begin{equation}\label{charact2}
\frac{dx}{dt} = a(x,t), \qquad x(d)=0.
\end{equation}
To avoid excessive repetition, we outline the method only in the case for $d>0$.
The solution $u$ can be decomposed into the sum of a regular component $v \in C^2(\bar \Omega)$ and a layer component $w\in C^2(\bar\Omega)$, defined as the solutions of
\begin{subequations}\label{decomp*}
\begin{align}
Lv &= f, \quad (x,t)  \in \Omega; \quad v(0,t)=\psi _1(t),\ v(x,0)= u(x,0),\\
Lw &= 0, \quad (x,t)  \in\Omega;\quad w(0,t)=\psi _2(t;\ve) ,\ w(x,0) = 0.
\end{align}
\end{subequations}
 The non-singularly component $v$ is approximated with a classical scheme on a uniform mesh. We examine now the singular component $w$. Note that $w(x,t) \equiv 0, x \geq g_1(t+d; d,0)$.
Consider the transformed variables, $\widehat u(x,\tau ) = u(x,t)$ where
$
\tau := t- g^{-1}_1(x; d,0),
$
then
\[
\left [1-\frac{\widehat a(x,\tau)}{\widehat a(x,0)} \right] \widehat w_\tau+ \widehat a \widehat w_ x + \widehat b\widehat w =0, \quad \widehat w(0,\tau) = \psi _2(\tau+d).
\]
 To generate a parameter-uniform approximation to $w$,  solve the problem:
\begin{subequations}\label{scheme-P-singular-related}
\begin{align}
&\left [1-\frac{\widehat a(x_i,\tau _j)}{\widehat a(x_i,0)}\right] D^*_\tau \widehat W+  \widehat a D_x^- \widehat W +\widehat b\widehat W= 0, \ (x_i,\tau _j)  \in \widehat S_\sigma:= (0,L] \times  (-\sigma,\sigma );\\
&
\widehat W(x_i,-\sigma) = \widehat W(x_i,\sigma) =0; \quad \widehat W(0, \tau _j) =\widehat \psi _2(\tau _j).
\end{align}
\end{subequations}
We  form  the global approximation $\bar U=\bar V + \bar W$ and
\[
\Vert u - \bar U \Vert   \leq C N^{-1} + CM^{-1}(\ln M)^2,
\]
for the problem class (\ref{main-problem}), (\ref{compat}), (\ref{pulse2}).

\end{document}